\documentclass[12pt]{article}
\usepackage{amsmath}
\usepackage{amssymb,latexsym}

\newcommand{\R}{\mathbb R}

\newcommand{\N}{\mathbb N}
\newcommand{\C}{\mathbb C}

\newenvironment{proof}{\noindent{\it Proof}\rm.}{\hfill $\Box$}
\newenvironment{proofof}[1]{\bigskip\noindent{\it Proof of~#1.}\rm}{\hfill $\Box$}
\newtheorem{theorem}{Theorem} [section]
\newtheorem{lemma}{Lemma} [section]

\newtheorem{corollary}{Corollary} [section]

\newtheorem{remark}{Remark}[section]
\begin{document}

\title{Universal inequalities for the eigenvalues of Laplace and 
Schr\"{o}dinger operators on submanifolds}
\author{Ahmad El Soufi$^1$ \and Evans M. Harrell II$^2$ \and Sa\"{\i}d Ilias$^1$}
\date{\small 
$^1$Universit\'e Fran{\c c}ois Rabelais de Tours, Laboratoire de Math\'ematiques
et Physique Th\'eorique, UMR-CNRS 6083, Parc de Grandmont, 37200
Tours, France\\
{\tt elsoufi@univ-tours.fr; ilias@univ-tours.fr}\\
$^2$School of Mathematics,
Georgia Institute of Technology,\\
Atlanta GA 30332-0160, USA\\
{\tt harrell@math.gatech.edu}}

\maketitle

\begin{abstract}
{\footnotesize
We establish inequalities for the eigenvalues of Schr\"{o}dinger operators on compact submanifolds (possibly with nonempty boundary) of Euclidean spaces, of spheres, and of real, complex and quaternionic projective spaces, which are related to inequalities  for the Laplacian on Euclidean domains due to
Payne, P\'olya, and Weinberger and to Yang, but which depend in an
explicit way on the mean curvature.  In later sections, we prove similar results for Schr\"{o}dinger operators on homogeneous Riemannian spaces and, more generally, on any Riemannian manifold that admits an eigenmap into a sphere,
as well as for the Kohn Laplacian on subdomains of the Heisenberg group.

Among the consequences of this analysis are an extension of Reilly's inequality, bounding any eigenvalue of the Laplacian in terms of the mean curvature, and spectral criteria for the immersibility of
manifolds in homogeneous spaces.



}
\end{abstract}

\section{Introduction}

Universal eigenvalue inequalities date from the work of Payne, P\'olya, and Weinberger in the 1950's \cite{PPW}, who considered the Dirichlet problem for the Laplacian on a Euclidean domain.  In this and similar cases, the term ``universal'' applies to expressions involving only the eigenvalues of a class of operators, without reference to the details of any specific operator in the class.  Since that time the essentially purely algebraic arguments that lead to universal inequalities have been adapted in various ways for eigenvalues of differential operators on manifolds e.g., see \cite{AB,CY1,CY2,H,HM, Le, L,NZ,T, YY}).  For a review of universal eigenvalue inequalities, we refer to \cite{A,AH1}.)  In particular, Ashbaugh and Benguria discussed universal inequalities for Laplacians
on subdomains of hemispheres in \cite{AB}, and
Cheng and Yang have treated the case of Laplacians on minimal submanifolds of spheres \cite{CY1}.

When either the geometry is more complicated or a potential energy is introduced, analogous inequalities must contain appropriate modifications.  Our point of departure is a recent article \cite{H}, in which the eigenvalues of  Schr\"odinger operators on hypersurfaces  were  studied and some trace identities and sharp inequalities were presented, containing the mean curvature explicitly.  The goal of the present article is to further study the relation between the spectra of Laplacians or
Schr\"odinger operators and the local differential geometry of submanifolds of arbitrary codimension.  
The approach is based on an algebraic technique which allows us to unify and 
extend many results in the literature (see \cite{A,AH1,AH2,AH3,H,HM,Le,L,NZ,PPW,Y,YY} 
and  Remarks~\ref{rk}, \ref{rk3} \ref{rk4}).
There is an extension of the results of \cite{H} to the case of submanifolds of codimension greater than one, and because of the appearance of the mean curvature, we are able to generalize Reilly's inequality \cite{R,ElIl86,ElIl92,ElIl00} by bounding each eigenvalue of the Laplacian in terms of the mean curvature.
In addition we derive the modifications necessary when the domain is contained in a submanifold of spheres, projective spaces, and certain other types of spaces.  Finally,
we are able to obtain some universal inequalities in the rather different context of the Kohn Laplacian on subdomains of the Heisenberg group.

Let us point out the following phenomenon which appears as a particular case of our results in Section 2 : For any compact submanifold $M$ of a Euclidean space, the eigenvalues of the operator $$-\Delta +  \frac{|h|^{2}} 4,$$
where $h$ is the mean curvature vector field of $M$, satisfy exactly the same universal inequalities  of PPW, HP and Yang type, as those satisfied by the eigenvalues of the Dirichlet Laplacian of a Euclidean domain.  This result is to be compared with the fact that when we consider the Laplace operator $-\Delta$ all alone (that is without the geometric potential term), then any finite sequence of positive numbers can be realized as the beginning of the spectrum of $-\Delta$ on a compact submanifold of a Euclidean space with given topology (indeed, this is a direct consequence of  the well-known construction of Colin de Verdi{\`e}re \cite{CV} and the famous Nash-Moser isometric embedding theorem). This means that there exist no universal inequalities for the  eigenvalues of the Laplace operator on compact submanifolds. Roughly speaking, one can say that, while the spectral behavior of the Laplace operator on compact submanifolds is ``unforseable'', the spectral behavior of the operator $-\Delta +  \frac{|h|^{2}} 4$ is as rigid as the Dirichlet Laplacian on Euclidean domains. 

The existence of universal eigenvalue inequalities appears at first to run counter to the well-known construction of Colin de Verdi{\`e}re \cite{CV}, allowing one to specify an arbitrary finite 
number of eigenvalues of 
a Laplacian or a Schr\"odinger operator if one is free to choose the metric or 
the potential 
energy on a manifold.   
From that point of view, universal eigenvalue inequalities like the ones in this article
either imply bounds on the 
potential energy in relation to the mean curvature, or else
necessary conditions for the embeddability of the Colin de Verdi{\`e}re examples 
as submanifolds of 
Euclidean or other symmetric spaces.

Let $M^{n}$ be a compact Riemannian manifold of dimension $n$, possibly with nonempty boundary $\partial M$, and let $\Delta $ be the Laplace--Beltrami operator on $M$.  In the case where $\partial M \neq \emptyset$, Dirichlet boundary conditions apply (in the weak sense \cite{D}).
For any bounded real--valued potential $q $ on $M$, the Schr\"{o}dinger operator
$$H= - \Delta+q$$
has compact resolvent (see \cite[Theorem~IV.3.17]{K} and observe that a bounded $q$ is relatively compact
with respect to $\Delta$).
The spectrum of  $H$
consists of a nondecreasing, unbounded sequence of eigenvalues with finite multiplicities \cite{Cv,D}:
$$Spec(-\Delta+q)=\{\lambda_1  <  \lambda_2 \le \lambda_3\le
\cdots \le \lambda_i\le \cdots \}.$$
Notice that when $\partial M = \emptyset$ and $q=0$, the zero eigenvalue is indexed by 1, that is,
$\lambda_1=0$.  The $L^2$-normalized eigenfunctions will be denoted 
$\left\{u_i\right\}$, so that $H u_i = \lambda_i u_i$.

To avoid technicalities, we suppose throughout that $q$ is bounded, and that the mean curvature of the submanifolds under consideration is defined everywhere and bounded.  Extensions to a wider class of potentials and geometries allowing singularities would not be difficult.

\section {Submanifolds of $\R^m$}

In this section $M$ is either a closed Riemannian manifold or a bounded domain in a Riemannian manifold that can be immersed as a  submanifold of dimension $n$ of $\R^m$.  The main theorem
directly extends a result of \cite{H}, in which
part (I) descends ultimately from a result of H.~C.~Yang for Euclidean domains \cite{Y,HS,AH1,AH2}:

\begin {theorem} \label{PPW euclidean}
Let $X: M \longrightarrow {\R}^{m}$ be an isometric immersion. We denote by $h$ the mean curvature vector field of $X$ (i.e the trace of its second fundamental form). For any bounded potential $q$ on $M$, the spectrum of $H=-\Delta+q$
(with Dirichlet boundary conditions if $\partial M \neq \emptyset$)
must satisfy, $\forall k\ge1$,
\begin{enumerate}
\item[(I)] $\displaystyle{n \sum_{i=1}^{k}(\lambda_{k+1}-\lambda_{i})^{2} \le 4 \sum_{i=1}^{k}(\lambda_{k+1}-\lambda_{i})\left(\lambda_{i}+\delta_i \right)}$
\item[(II)]$\displaystyle{\left( 1+\frac{2}{n}\right) \frac{1}{k} 
\sum_{i=1}^{k}{\lambda_{i}}+ \frac{2}{n}
 \frac{1}{k} \sum_{i=1}^{k}{\delta_i}
-\sqrt{D_{nk}} \le \lambda_{k+1}}$
\item[\ \ \ \ \ ]$\ \ \ \ \ \ \ \ \ \ \ \ \
\displaystyle\le \left( 1+\frac{2}{n}\right) \frac{1}{k} \sum_{i=1}^{k}{\lambda_{i}}+ \frac{2}{n}
 \frac{1}{k} \sum_{i=1}^{k}{\delta_i}+\sqrt{D_{nk}},$
\end{enumerate}
where $u_i$ are the $L^2$--normalized eigenfunctions,
$\delta_i := \int_{M}\left({{|h|^{2}} \over 4}-q\right)u_{i}^{2}$, and
$$
{\rm (III)}\quad D_{nk} := \left({\left( 1+\frac{2}{n}\right) \frac{1}{k}  \sum_{1}^{k}{\lambda_{i}} +  \frac{2}{n}
 \frac{1}{k} \sum_{i=1}^{k}{\delta_i} }\right)^2 - \left( 1+\frac{4}{n}\right)  \frac{1}{k}\sum_{1}^{k}{\lambda_{i}^2} $$
 \hfill $\displaystyle - \frac{4}{n}
 \frac{1}{k}\sum_{i=1}^{k}{\lambda_{i} \delta_{i}}  \ge 0.
$
\end{theorem}

Theorem \ref{PPW euclidean} can be simplified to eliminate all dependence on
$u_i$ with elementary estimates such as

$$\inf{\left({{|h|^{2}} \over 4}-q\right)} \le \delta_i \le \sup{\left({{|h|^{2}} \over 4}-q\right)}. \eqno{(2.1)}$$
Thus:

\begin{corollary} \label{coro1}
Under the circumstances of Theorem \ref{PPW euclidean}, $\forall k\ge 1$,
\begin{enumerate}
\item[(I a)] $\displaystyle{n \sum_{i=1}^{k}(\lambda_{k+1}-\lambda_{i})^{2} \le 4 \sum_{i=1}^{k}(\lambda_{k+1}-\lambda_{i})\left(\lambda_{i}+\delta\right)}$
\item[(II a)]$\displaystyle{\lambda_{k+1} \le \left(1+\frac{4}{n}\right) \frac{1}{k} \sum_{i=1}^{k}{\lambda_{i}}+\frac{4 \delta}{n}.}$
\end{enumerate}
where $\delta:=\sup{\left({{|h|^{2}} \over 4}-q\right)}  $.
\end{corollary}
\par
\noindent
Corollary~\ref{coro1}, proved below, can be restated as a criterion for the immersibility of a manifold in $\R^m$:

\begin{corollary} \label{immers1}
Suppose that $\left\{\lambda_{i}\right\}$ are the eigenvalues of the Laplace--Beltrami operator on an abstract compact Riemannian manifold $M$ of dimension $n$.
If $M$ is isometrically immersed in $\R^m$, then the mean curvature satisfies
$$\|h\|_{\infty}^2 \geq n \lambda_{k+1} - \frac{\left(n+4\right)}{k} \sum_{i=1}^{k}{\lambda_{i}}\eqno{(2.2)}$$
for each $k$.
\end{corollary}
\par\noindent
Corollary~\ref{immers1} is representative of a large family of necessary conditions for immersibility in terms of the eigenvalues of  Laplace--Beltrami and Schr\"odinger operators on $M$, which will not be presented in detail in this article.  (See \cite{H} for various sum rules on which such constraints can be based.)

\begin{proofof}{Theorem~\ref{PPW euclidean}}
For a smooth function $G$ on $M$, we will denote by $G$ the multiplication operator naturally associated with $G$. To prove Theorem~\ref{PPW euclidean} we first need the following lemma involving the commutator of $H$ and $G$, $\left[H, G\right] := H G - G H$.

\begin{lemma}\label{commutation}
For any smooth $G$ and any positive integer $k$ one has
$$\sum_{i=1}^{k}(\lambda_{k+1}-\lambda_{i})^{2} \langle[H,G]u_{i},Gu_{i}\rangle_{L^{2}} \le \sum_{i=1}^{k}(\lambda_{k+1}-\lambda_{i})\left\|[H,G]u_{i}\right\|_{L^2}^{2}\eqno{(2.3)}$$
\end{lemma}

This lemma dates from \cite[Theorem~5]{HS}, and in this form appears in 
\cite[Theorem~2.1]{AH1}.
Variants can be found in \cite[Corollary~4.3]{H} and \cite[Corollary~2.8]{LP}.

Now, let $X_1, \dots, X_m$ be the components of the immersion $X$. A straightforward calculation gives
$$[H,X_{\alpha}]u_{i}=[-\Delta,X_{\alpha}] u_{i}=(-\Delta X_{\alpha})u_{i}-2\nabla X_{\alpha}\cdot\nabla u_{i}.$$
It follows by integrating by parts that
$$ \langle[H,X_{\alpha}]u_{i},X_{\alpha}u_{i}\rangle_{L^2}=\int_{M} |\nabla X_{\alpha}|^{2} u_{i}^{2}.$$
Thus $$ \sum_{\alpha} \langle[H,X_{\alpha}]u_{i},X_{\alpha}u_{i}\rangle_{L^2}= \sum_{\alpha}\int_{M} |\nabla X_{\alpha}|^{2}u_{i}^{2}=n \int_M u_{i}^{2}=n.$$
On the other hand, we have
$$\left\|[H,X_{\alpha}]u_{i}\right\|_{L^2}^{2}= \int_{M} \left((-\Delta X_{\alpha})u_{i}-2\nabla X_{\alpha}\cdot\nabla u_{i})\right)^{2}.$$
Since $X$ is an isometric immersion, it follows that 
$h=(\Delta X_1,\dots,\Delta X_m)$, 
$\sum_{\alpha}\left(\nabla{X_{\alpha}}\cdot\nabla u_{i}\right)^{2}=|\nabla u_{i}|^{2}$ and $\sum_{\alpha}(-\Delta X_{\alpha})u_{i}\nabla X_{\alpha}\cdot\nabla u_{i} = h\cdot\nabla u_i^2=0$.
Using all these facts,  we get
$$\sum_{\alpha}\left\|[H,X_{\alpha}]u_{i}\right\|_{L^2}^{2}= \int _{M} |h|^{2}u_{i}^{2} + 4 \int_{M} |\nabla u_{i}|^{2},\eqno{(2.4)}$$
as in \cite{H}.  Then
\begin{eqnarray}
\nonumber{} \int_{M} |\nabla u_{i}|^{2} &=&  \int_{M} u_{i}(-\Delta+q)u_{i}- \int_{M} qu_{i}^{2}\\
   \nonumber{} &=& \left(\lambda_{i}-\int_{M}qu_{i}^{2}\right).
\end{eqnarray}
Using Lemma \ref{commutation} we obtain
$$n \sum_{i=1}^{k} (\lambda_{k+1}-\lambda_{i})^{2} \le \sum_{i=1}^{k} (\lambda_{k+1}-\lambda_{i})\left(\int_{M} (|h|^{2}-4q)u_{i}^{2} + 4\lambda_{i}\right)$$
which proves assertion (I) of Theorem~\ref{PPW euclidean}.
\par
\noindent
From assertion (I) we get a quadratic inequality in the variable $\lambda_{k+1}$:
\begin{align*}
k \lambda_{k+1}^2 &-  \lambda_{k+1} \left(\left( 2+\frac{4}{n}\right)  \sum_{i=1}^{k}{\lambda_{i}}
+\frac{4}{n} \sum_{i=1}^{k}{\delta_{i}}\right)\\
& + \left( 1+\frac{4}{n}\right) \sum_{i=1}^{k}{\lambda_{i}^2}
+ \frac{4}{n} \sum_{i=1}^{k}{\lambda_{i} \delta_{i}} \le 0\tag{2.5}\end{align*}
The roots of this quadratic polynomial are the bounds in (II).  The existence and reality of
$\lambda_{k+1}$ imply statement (III).
\end{proofof}

\begin{proofof}{Corollary~\ref{coro1}}
To derive (II~a) from Theorem~\ref{PPW euclidean}(II), it is simply necessary to replace $\delta_i$ by $\delta$, and to note that
the quantity $D_{nk}$
is bounded above by
$$\left(\left( 1+\frac{2}{n}\right) \frac{1}{k} \sum_{i=1}^{k}{\lambda_{i}} + \frac{2 \delta}{ n}\right)^2 -  \left( 1+\frac{4}{n}\right) \frac{1}{k} \sum_{i=1}^{k}{\lambda_{i}^2} - \frac{4 \delta}{n} \frac{1}{k} \sum_{i=1}^{k}{\lambda_{i}},$$
which, since $\left( \sum_{i=1}^{k}{\lambda_{i}} \right) \le k \sum_{i=1}^{k}{\lambda_{i}^2} $, 
implies that
$$D_{nk} \le \left(\frac{2}{n} \frac{1}{k} \sum_{i=1}^{k}{\lambda_{i}}\right)^2 +  \left(\frac{2 \delta}{n}\right)^2 +
\frac{8 \delta}{n^2} \frac{1}{k}\sum_{i=1}^{k}{\lambda_{i}} = \left({\frac{2}{n} \frac{1}{k} \sum_{i=1}^{k}{\lambda_{i}} + \frac{\delta}{2n}}\right)^2,$$
with which the upper bound in (II) reduces to the right member of (II~a).
\end{proofof}

\medskip

We observe next
that Theorem~\ref{PPW euclidean}  enables us to recover Reilly's inequality
for $\lambda_2$ of the Laplace--Beltrami operator on closed submanifolds \cite{ElIl92,R}. Indeed, applying (I) with $k=1$, $\lambda_1=0$ and $u_1=V^{-\frac 1 2}$, where $V$ is the volume of $M$, we get
$$ \lambda_2\le \frac 4 n \ \delta_1=\frac 1 {nV}\int_M |h|^2\le \frac 1 n \left\|h\right\|_{\infty}^{2}.$$
Moreover, Theorem~\ref{PPW euclidean} allows extensions of Reilly's inequality to higher order eigenvalues. For example,  the following corollary can be derived easily from Corollary~\ref{coro1}(II~a) by induction on $k$.
\begin{corollary} \label{Reilly}
Under the circumstances of Theorem~\ref{PPW euclidean}, $\forall k\ge 2$,
$$\lambda_{k} \le \left(\frac 4 n +1\right)^{k-1} \lambda_1 +C_{R}(n,k) \left\|h\right\|_{\infty}^{2},$$
where $C_R(n,k) = \frac 1 4 \left((\frac 4 n +1)^{k-1} -1 \right)$.
In particular, when $M$ is closed and $q=0$,
$$\lambda_{k} \le C_{R}(n,k) \left\|h\right\|_{\infty}^{2}.\eqno{(2.6)}$$

\end{corollary}

\par\noindent
The explicit value for the generalized Reilly constant $C_{R}(n,k)$ given in this 
corollary  is likely  far from optimal.
We regard the
sharp value of $C_{R}(n,k)$ as an interesting open problem. 
In the case of a minimally embedded submanifold of a sphere,
Cheng and Yang state a bound on  $\lambda_{k}$ (\cite{CY1}, eq. (1.23)) that scales
like $k^{2 \over n}$ as in the Weyl law.  We conjecture
that $C_{R}(n,k)$ is sharply bounded by a constant times $ k^{2 \over n}$ when $q = 0$ and that when $q \ne 0, C_{R}(n,k)$ is
correspondingly bounded by a semiclassical expression, as is the case for
Schr\"odinger operators on flat spaces.  (See, for instance,
\cite{ThIII}, section~3.5 and \cite{Lieb}, part III.)

In \cite{H} it was argued that simplifications
and optimal inequalities are obtained in some circumstances where
${M}$ is a hypersurface and the potential
$q$ depends quadratically on curvature, a circumstance that arises naturally in the physics of thin structures (\cite{ExSe,ExHaLo} and references therein).  In this
spirit we close the section with some
remarks for Schr\"odinger operators $H_g := - \Delta + g |h|^2,$ for a real parameter $g$.  As was already observed in \cite{H}, in view of (2.1), simplifications occur
when $g = {1 \over 4}$, rendering the quantities $\delta$ and $\delta_j$ given above zero.

\begin{corollary} \label{optimal}
Assume $M$ is closed, $|h|$ is bounded, and $H$ is of the form
$H_g$, where $g$ is an arbitrary real number.
The inequalities {\rm (I)}, {\rm (II)}, and {\rm (III)},  in Theorem~\ref{PPW euclidean}
are saturated (i.e., equalities) for all $k$  
if $M$ is a sphere.
\end{corollary}

\begin{proof}
We begin with the case of the Laplacian, $g=0$, for which the eigenvalues of the standard sphere
$\mathbb{S}^{n}$ are known
\cite{M} to be $\left\{\ell (\ell + n - 1)\right\}$, $\ell = 0, 1, \dots$, with multiplicities 1 for 
$\ell=0$; $n+1$ for $\ell=1$;
and $\mu_{n, \ell} := {{n+\ell} \choose n} - {{n+\ell-2} \choose n} = {{(n(n+1) \dots (n + \ell - 2))(n+2 \ell-1)} \over {\ell!}}$ thereafter.  
Thus $\lambda_1 = 0$, $\lambda_2 = \dots = \lambda_{n+2} = n$, etc., with gaps separating 
eigenvalues $\lambda_k$ and $\lambda_{k+1}$ when $k = \sum_{\ell=0}^m{\mu_{n, \ell}} = {n+2m \over n}{{n+m-1}\choose{m} }$.  For this corollary it suffices to consider only the values of $k$ at gaps 
such that $\lambda_{k+1} > \lambda_{k}$, because for any $k$ such that 
$\lambda_{k+1} = \lambda_{k}$, the two sides of  Inequalities {\rm (I)} are the same as for the next lower 
value $k_-$ such that $\lambda_{k_-} < \lambda_{k_- + 1} = \lambda_{k}$; the additional contributions are all equal to $0$.

For the sphere,
$\delta_j = {n^2 \over 4}$, and
an exact calculation shows, remarkably, that
$$n \sum_{i=1}^{k}\left(\lambda_{k+1}^{sphere}-\lambda_{i}^{sphere}\right)^{2} 
= \sum_{i=1}^{k}\left(\lambda_{k+1}^{sphere}-\lambda_{i}^{sphere}\right)
\left(4 \lambda_{i}^{sphere}+n^2 \right):$$
To see this, subtract 
$n \sum_{i=1}^{k}\left(\lambda_{k+1}^{sphere}-\lambda_{i}^{sphere}\right)^{2}$ 
from the expression on the right and multiply the 
result by $(n-1)!$.  After substitution and simplification, the expression reduces to
$$
\sum_{\ell=1}^{m}{{\scriptstyle{ (m-\ell+1) (n+m+ \ell)(2 \ell +n-1)(4 \ell (\ell-1) - n^2 (m-\ell) - n (m^2+m-\ell(\ell+3)) (n+ \ell-2)!)}\over{\ell !}} },$$
which evaluates identically to $0$.  (Algebra was performed with the aid of 
Mathematica\texttrademark.)

This establishes equality in (I), and consequently (II) and (III) for this case.  If $M = \mathbb{S}^{n}$,
$|h|^2 = n^2$ is a constant, and if
$g n^2$ is added to $-\Delta$, then each eigenvalue is shifted
by the same amount and the left side of (I) is unchanged, as is the first factor
in the sum on the right.  As for the other factor, it becomes
$\lambda_{i} + \delta_j = \ell (\ell + n - 1) + g n^2 + {n^2 \over 4} - g n^2$
and is likewise unchanged.  It follows that the case of equality
for $H_g$ on the standard sphere persists for all $g$.
\end{proof}

 \section {Submanifolds of spheres and projective spaces}

\medskip

Theorem~\ref{PPW euclidean}, together with the standard embeddings of sphere and projective spaces by means of the first eigenfunctions of their Laplacians, enables us to obtain results for immersed submanifolds of the latter. In what follows, $\mathbb{F}$ will denote the field $\mathbb{R}$ of real numbers, the field $\mathbb{C}$ of complex numbers, or the field $\mathbb{Q}$ of quaternions. The $m$-dimensional projective space over $\mathbb{F}$ will be denoted by $\mathbb{F}P^{m}$ ; we endow it with its standard Riemannian metric so that the sectional curvature is either constant and equal to 1 ($\mathbb{F}=\R$) or pinched between 1 and 4 ($\mathbb{F}=\C$ or $\mathbb{Q}$). For convenience, we
introduce the integers
\begin{equation*}
d(\mathbb{F})= \dim_\R \mathbb{F}=
   \begin{cases}
   1      &\text{if $\mathbb{F}=\mathbb{R}$}\\
   2      &\text{if $\mathbb{F}=\mathbb{C}$}\\
   4      &\text{if $\mathbb{F}=\mathbb{Q}$.}
   \end{cases}
\end{equation*}
and
\begin{equation*}
c(n)=
    \begin{cases}
    n^{2},     &\text{if $\overline{M}=\mathbb{S}^{m}$}\\
    2n(n+d(\mathbb{F})),   &\text{if $\overline{M}= \mathbb{F}P^{m}.$}
    \end{cases}\eqno{(3.1)}
\end{equation*}

\begin {theorem} \label{PPW symmetric}
Let $\overline{M}$ be $\mathbb{S}^{m}$ or $\mathbb{F}P^{m}$ and let $X: M \longrightarrow \overline{M}$ be an isometric immersion of mean curvature $h$. For any bounded potential $q$ on $M$, the spectrum of $H=-\Delta_g+q$ (with Dirichlet boundary conditions if $\partial M \neq \emptyset$) must satisfy, $\forall k\in \N$, $k\ge1$,
\begin{enumerate}
\item[(I)]$\displaystyle{n \sum_{1}^{k}(\lambda_{k+1}-\lambda_{i})^{2} \le 4\sum_{i=1}^{k}(\lambda_{k+1}-\lambda_{i})\left(\lambda_{i}+\bar\delta_i\right),}$\\
where $\bar\delta_i:=\frac 1 4\int_{M}(|h|^{2}+c(n)-4q)u_{i}^{2}$,
\item[(II)]$\displaystyle{\lambda_{k+1} \le \left( 1+\frac{2}{n}\right) \frac{1}{k}  \sum_{i=1}^{k} \lambda_{i}+ \frac 2 n \frac 1 k \sum_{i=1}^{k} \bar\delta_{i}+\sqrt{\bar D_{nk}}}$
\end{enumerate}
where $$ \bar D_{nk}:= \left(\left(1 + \frac{2}{n}\right) \frac{1}{k} \sum_{1}^{k}{\lambda_{i}} +
\frac 2 n \frac 1 k \sum_{i=1}^{k} \bar\delta_{i}\right)^2$$
$$\quad\quad\quad\quad\quad- \left(1 + \frac{4}{n}\right) \frac{1}{k} \sum_{1}^{k}{\lambda_{i}^2}
- \frac 4 n \frac 1 k \sum_{i=1}^{k} \lambda_i\bar\delta_{i} \ge 0,$$
\end {theorem}
\par\noindent
A lower bound is also possible along the lines of Theorem~\ref{PPW euclidean}.  As in
the previous section,
the following simplifications follow easily:

\begin {corollary} \label{coro3}
With the notation of Theorem~\ref{PPW symmetric} one has,
 $\forall k\ge1$,
$$\displaystyle \lambda_{k+1} \le \left( 1+\frac{4}{n}\right) \frac{1}{k}  \sum_{i=1}^{k} \lambda_{i}+
\frac{4}{n}\ \bar\delta, $$
where $\bar\delta:=\frac1 4{\sup{(|h|^{2}+c(n)-4q)}}.$
\end{corollary}

Moreover, as in the discussion for Corollary~\ref{optimal},
when $M$ is a submanifold of a sphere or projective space, a simplification
occurs in Theorem~\ref{PPW symmetric} and Corollary~\ref{coro3}
when $q(x) = \frac1{4}({|h|^2 + c(n)})$, in that the curvature and potential
do not appear explicitly at all.

\begin{remark}\label{rk}
Theorems~\ref{PPW euclidean} and \ref{PPW symmetric} and Corollaries \ref{coro1} and \ref{coro3} unify and extend many results in the literature (see \cite{A,AH1,AH2,AH3,H,HM,Le,L,NZ,PPW,Y,YY} and the references therein). In particular, the recent results of  Cheng and Yang \cite{CY1} and \cite{CY2} concerning the eigenvalues of the Laplacian on
\begin{itemize}
       \item [-] a domain or a minimal submanifold of  $\mathbb S^m$
       \item [-] a domain or a complex  hypersurface of $\C P^m$
\end{itemize}
respectively, appear as particular cases of Theorem~\ref{PPW symmetric}. Recall that a complex submanifold of $\C P^m$ is automatically minimal (that is, $h=0$).

\end{remark}

\begin{proofof}{Theorem~\ref{PPW symmetric}}
We will treat separately the cases $\overline{M}=\mathbb{S}^{m}$ and  $\overline{M}=\mathbb{F}P^{m}$.
\medskip

\noindent \underline{Immersed submanifolds of a sphere}:
\medskip

Let $\bar{M}= \mathbb{S}^{m}$ and denote by $i$ the standard embedding of $\mathbb{S}^{m}$ into $\mathbb{R}^{m+1}$.
We have
$$ |h(i\circ X)|^{2}=|h(X)|^{2}+n^{2}.$$
Applying Theorem~\ref{PPW euclidean} to the isometric immersion $i \circ X:(M,g)\to\R^{m+1}$, we obtain the result.

\medskip
\noindent \underline{Immersed submanifolds of a projective space}: 
\medskip

First, we need to recall some facts about the first standard embeddings of projective spaces into Euclidean spaces (see for instance \cite{C,S,T} for details).
Let $\mathcal{M}_{m+1}(\mathbb{F})$ be the space of $(m+1)\times(m+1)$ matrices over $\mathbb{F}$ and set $\mathcal{H}_{m+1}(\mathbb{F})=\{ A \in \mathcal{M}_{m+1}(\mathbb{F})\; |\; A^{\ast}:= ^{t}{\bar A}=A \}$ the subspace of Hermitian matrices.
We endow $\mathcal{M}_{m+1}(\mathbb{F})$ with the inner product given by
     $$ \langle A,B\rangle= \frac{1}{2}  tr (A \, B^{\ast}).$$
For $A,\,B \in \mathcal{H}_{m+1}(\mathbb{F})$, one simply has $ \langle A,B\rangle= \frac{1}{2} tr (A \, B).$

The first standard embedding $\varphi: \mathbb{F}P^{m}\to \mathcal{H}_{m+1}(\mathbb{F})$ is defined as the one induced via the canonical fibration $\mathbb{S}^{(m+1)d-1} \to \mathbb{F}P^{m}$ ($d:=d(\mathbb{F})$), from the natural immersion $ \psi: \mathbb{S}^{(m+1)d-1} \subset \mathbb{F}^{m+1}\longrightarrow \mathcal{H}_{m+1}(\mathbb{F})$
given by
$$\psi(z)=
\begin{pmatrix}
|z_{0}|^{2} & z_{0} \bar{z}_{1} & \cdots & z_{0}\bar{z}_{m}\\
z_{1}\bar{z}_{0} & |z_{1}|^{2} & \cdots & z_{1}\bar{z}_{m} \\
 \cdots        & \cdots             & \cdots & \cdots \\
z_{m}\bar{z}_{0}  &  z_{m}\bar{z}_{1} & \cdots & |z_{m}|^{2}
\end{pmatrix}.
$$
The embedding $\varphi$ is isometric and the components of $\varphi - \frac 1 {m+1} I$ are eigenfunctions associated with the first eigenvalue of the Laplacian of $\mathbb{F}P^{m}$ (see, for instance, \cite{T} for details).  Hence, $\varphi(\mathbb{F}P^{m})$ is a minimal submanifold of the hypersphere
 $\mathbb{S}\left(\sqrt{m/2(m+1)}\,\right)$ 
of $\mathcal{H}_{m+1}(\mathbb{F})$ centered at $\frac{1}{m+1}I$.

\begin{lemma}\label{mean}
Let $X:M\to \mathbb{F}P^{m}$ be an isometric immersion and let $h$ and $h'$ be the mean curvature vector fields of the immersions $X$ and $\varphi \circ X$ respectively. Then we have
$$ |h'|^{2}=|h|^{2}+ \frac{4n(n+2)}{3}+ \frac{2}{3} \sum_{i \ne j} K(e_{i},e_{j})$$
where $K$ is the sectional curvature of $\mathbb{F}P^{m}$ and $(e_{i})_{i\le n}$ is a local orthonormal frame tangent to  $X(M)$.
\end{lemma}
 We refer to \cite{C,S}, or \cite{T} for a proof of this lemma.

Now, from the expression of the sectional curvature of $\mathbb{F}P^{m},\forall i \ne j$
we get
\begin{itemize}
\item $K(e_{i},e_{j})=1$ if $\mathbb{F}=\mathbb{R}$.
\item $K(e_{i},e_{j})=1+3 \left(e_{i}\cdot Je_{j}\right)^{2}$, where $J$ is the almost complex structure of $\mathbb{C}P^{m}$, if $\mathbb{F}=\mathbb{C}$.
\item $K(e_{i},e_{j})=1+\sum_{r=1}^{3}3\left( e_{i}\cdot J_{r}e_{j}\right)^{2}$, where $(J_{1},J_{2},J_{3})$ is the almost quaternionic structure of $\mathbb{Q}P^{m}$, if $\mathbb{F}=\mathbb{Q}$.
\end{itemize}
Thus in the case of $\mathbb{R}P^{m}$, we obtain $ |h'|^{2}=|h|^{2}+2n(n+1)$. For $\mathbb{C}P^{m}$, we get
\begin{eqnarray}
\label{J}\nonumber{}
|h'|^{2}&=&|h|^{2}+2n(n+1)+2\sum_{i, j}\left(e_{i}\cdot Je_{j}\right)^{2}\\
\nonumber{} &=& |h|^{2}+2n(n+1)+2\left\|J^{T}\right\|^{2}\le |h|^{2}+2n(n+2),\quad (3.2)
\end{eqnarray}
where $J^{T}$ is the tangential part of the almost complex structure $J$ of $\mathbb{C}P^{m}$. Indeed, we clearly have $\left\|J^{T}\right\|^{2}\le n$, where the equality holds if and only if $X(M)$ is  a complex submanifold of  $\mathbb{C}P^{m}$.
For the case of $\mathbb{Q}P^{m}$, we obtain similarly
\begin{eqnarray}
\label{Jr}\nonumber{}
|h'|^{2}&=&|h|^{2}+2n(n+1)+2\sum_{i,j}\sum_{r=1}^{3}\left( e_{i}\cdot J_{r}e_{j}\right)^{2}\\
\nonumber{}&=&|h|^{2}+2n(n+1)+2\sum_{r=1}^{3}\|J_{r}^{T}\|^{2} \le |h|^{2}+2n(n+4), \quad (3.3)
\end{eqnarray}
where $(J_{r}^{T})_{1 \le r \le 3}$ are the tangential components of the almost quaternionic structure of $\mathbb{Q}P^{m}$. The equality in (3.3) holds if and only if $n \equiv 0$ (mod 4) and $X(M)$ is an invariant submanifold of $\mathbb{Q}P^{m}$.

To finish the proof of Theorem~\ref{PPW symmetric}, it suffices to apply Theorem~\ref{PPW euclidean} to the isometric immersion $\varphi \circ X$ of $M$ in the Euclidean space $\mathcal{H}_{m+1}(\mathbb{F})$ using the inequalities
(3.2) and (3.3).
\end{proofof}

\begin{remark}\label{rk2}
It is worth noticing that in some special geometrical situations, the constant $c(n)$ in the inequalities of Theorem~\ref{PPW symmetric} and Corollary~\ref{coro3} can be replaced by a sharper one. For instance, when $\bar M=\C P^m$ and
\begin{itemize}
       \item [-] $M$ is odd--dimensional, then one can replace  $c(n)$ by $c'(n)= 2n(n+2-\frac 1 n)$,
       \item [-] $X(M)$ is totally real (that is $J^T=0$), then  $c(n)$ can be replaced by $c'(n)= 2n(n+1)$.
       \end{itemize}
Indeed, under each one of these assumptions, the estimate of $\|J^T\|^2$ by $n$ (see the inequality (3.2) above) can be improved by elementary calculations.
\end{remark}

\section {Manifolds admitting spherical eigenmaps}
Let $(M,g)$ be a compact Riemannian manifold. A map $$\varphi=(\varphi_1\dots,\varphi_{m+1}):(M,g)\longrightarrow \mathbb{S}^{m}$$ is termed an \textit{eigenmap} if  its components $\varphi_1\dots,\varphi_{m+1}$ are all eigenfunctions associated with the same eigenvalue $\lambda$ of the Laplacian of $(M,g)$. Equivalently, an eigenmap is a harmonic map with constant energy density ($\sum_{\alpha}|\nabla\varphi_{\alpha}|^2=\lambda$) from  $(M,g)$ into a sphere. In particular, any minimal and homothetic immersion of $(M,g)$ into a sphere is an eigenmap. Moreover, a compact homogeneous Riemannian manifold without boundary admits eigenmaps for all the positive eigenvalues of its Laplacian (see for instance \cite{L}).

We still denote by $\left\{u_i\right\}$ a complete $L^2$-orthonormal basis of eigenfunctions of $H$ associated to $\left\{\lambda_i\right\}$.
\begin{theorem}\label{PPW eigenmap}
Let $\lambda$ be an eigenvalue of the Laplacian of $(M,g)$ and assume that $(M,g)$ admits an eigenmap associated with the eigenvalue $\lambda$. Then, for any bounded potential $q$ on $M$, the spectrum of $H=-\Delta_g+q$ (with Dirichlet boundary conditions if $\partial M \neq \emptyset$) must satisfy, $\forall k\in \N$, $k\ge1$,
\begin{enumerate}
\item[(I)] $\displaystyle{\sum_{1}^{k}(\lambda_{k+1}-\lambda_{i})^{2} \le \sum_{i=1}^{k}(\lambda_{k+1}-\lambda_{i})\left(\lambda +4 \left(\lambda_{i}-\int_{M}q u_{i}^{2}\right)\right)}$.
\item[(II)]$\displaystyle \lambda_{k+1} \le  ( 1+\frac{2}{n}) \frac{1}{k} \sum_{i=1}^{k} \lambda_{i}+ \frac{(\lambda-4 \inf q)}{2n}+\frac {\sqrt {\hat D}_{nk}}{2nk}$.
\end{enumerate}
where 
\begin{align*}
\hat D_{nk}&= \left(2(n+2)\sum_{1}^{k}\lambda_{i}+k(\lambda- \inf q)\right)^{2}\\
&\qquad -4nk \left((n+4)
\sum_{1}^{k}\lambda_{i}^{2}+(\lambda - \inf q)A\right)\end{align*}
\end{theorem}

\begin{corollary}\label{PPW homogeneous}
Let $(M,g)$ be a compact homogeneous Riemannian manifold without boundary. The inequalities of Theorem~\ref{PPW eigenmap} hold, $\lambda$ being here the first positive eigenvalue of the Laplacian of $(M,g)$.
\end{corollary}

\begin{remark}\label{rk3}
Theorem~\ref{PPW eigenmap} and Corollary~\ref{PPW homogeneous}  are to be compared to results of \cite{CY1,HM,L}.
\end{remark}

\begin{proofof}{Theorem~\ref{PPW eigenmap}}
Let $\varphi=(\varphi_1\dots,\varphi_{m+1}):(M,g)\to \mathbb{S}^{m}$ be a $\lambda$-eigenmap. As in the proof of Theorem~\ref{PPW euclidean}, we use Lemma \ref{commutation} with $G=\varphi_{\alpha}, \; \alpha=1,2,\dots,m+1$, to obtain
$$\sum_{\alpha}\sum_{i=1}^{k}(\lambda_{k}-\lambda_{i})^{2} \langle [H,\varphi_{\alpha}]u_{i},\varphi_{\alpha}u_{i}\rangle_{L^2} \le \sum_{\alpha}\sum_{i=1}^{k}(\lambda_{k}-\lambda_{i})\left\|[H,\varphi_{\alpha}]u_{i}\right\|_{L^2}^{2}.$$
A direct computation gives
$$ [H,\varphi_{\alpha}]u_{i}= \lambda\varphi_{\alpha} u_{i} - 2\nabla{\varphi_{\alpha}}\cdot\nabla{u_{i}}$$
and
$$
\langle[H,\varphi_{\alpha}]u_{i},\varphi_{\alpha}u_{i}\rangle_{L^2}=\lambda \int_{M} \varphi_{\alpha}^{2} u_{i}^{2} -\frac{1}{2} \int_{M} \nabla{\varphi_{\alpha}^2}\cdot\nabla{u_{i}^2}.$$
Summing up, we obtain
$$\sum_{\alpha} \langle[H,\varphi_{\alpha}]u_{i},\varphi_{\alpha}u_{i}\rangle_{L^2}=  \lambda,$$
since $\sum_{\alpha}\varphi_{\alpha}^2$ is constant.
Since $\sum_{\alpha}|\nabla\varphi_{\alpha}|^2=\lambda$ and $\int_{M}|\nabla{u_{i}}|^{2}=\lambda_{i}-\int_{M} q u_{i}^{2}$,
the same kind of calculation yields
\begin{align*}
 \sum_{\alpha}\left\|[H,\varphi_{\alpha}]u_{i}\right\|_{L^2}^{2}&=\lambda^{2}+4 
\sum_{\alpha} \int_{M} \left(\nabla{\varphi_{\alpha}}\cdot\nabla{u_{i}}\right)^{2} \\
 &\le  \lambda^{2} +4 \int_{M} \sum_{\alpha}|\nabla{\varphi_{\alpha}}|^{2}|\nabla{u_{i}}|^{2} \\
&= \lambda\left(\lambda + 4\left(\lambda_{i}-\int_{M} q u_{i}^{2}\right)\right).
\end{align*}
\par\noindent
In conclusion, we have
$$\lambda \sum_{1}^{k}(\lambda_{k+1}-\lambda_{i})^{2} \le \sum_{i=1}^{k}(\lambda_{k+1}-\lambda_{i})\left(\lambda^{2}
+4\lambda \left(\lambda_{i}-\int_{M}q u_{i}^{2}\right)\right),$$
which gives the first assertion of Theorem~\ref{PPW eigenmap}.
We derive the second assertion as in the proof of Theorem~\ref{PPW euclidean}.
\end{proofof}

\section {Applications to the Kohn Laplacian on the 
Heisenberg group}

Let us recall that the $2n+1$-dimensional Heisenberg group $\mathbb{H}^{n}$ is
the space $\mathbb{R}^{2n+1}$ equipped with the non-commutative group law
$$ (x,y,t)(x',y',t')=\left(x+x',y+y',t+t'+\frac{1}{2}\right)
(\left\langle x',y\right\rangle_{\mathbb{R}^{n}}-\left\langle x,y'\right\rangle_{\mathbb{R}^{n}}),$$
where $x,x',y,y'\in \mathbb{R}^{n},\; t \;\rm{and}\; t' \in \mathbb{R}$.
Its Lie algebra $\mathcal{H}^{n}$ has as a basis the vector fields
$$\left\{ T=\frac{\partial}{\partial t},\ X_{i}=\frac{\partial}{\partial x_{i}}+\frac{y_{i}}{2}\frac{\partial}{\partial t},\ Y_{i}=\frac{\partial}{\partial y_{i}}-\frac{x_{i}}{2}\frac{\partial}{\partial t}\ ; \ {i \leq n}\right\}.$$
We observe that the only non--trivial commutators are $\left[X_{i},Y_{j}\right]= - T \delta_{ij},\; i,j=1, \cdots ,n$.
Let $\Delta_{\mathbb{H}^{n}}$ denote the real Kohn Laplacian (or the sublaplacian associated with the basis 
$\left\{X_{1},\cdots,X_{n},Y_{1},\cdots,Y_{n}\right\}$): 
\begin{align*}
 \Delta_{\mathbb{H}^{n}}&= \sum_{i=1}^{n} X_{i}^{2}+Y_{i}^{2}\\
&= \Delta^{\mathbb{R}^{2n}}_{xy}+\frac{1}{4}(|x|^{2}+|y|^{2})\frac{\partial^{2}}{{\partial t}^{2}}+ \frac{\partial}{\partial t}\sum_{i=1}^{n}\left(y_{i}\frac{\partial}{\partial x_{i}}-x_{i}\frac{\partial}{\partial y_{i}}\right).\end{align*}
We shall be concerned with the following eigenvalue problem :

$$
-\Delta_{\mathbb{H}^{n}} {u} = \lambda {u}\ \hbox{in}\ \Omega
$$
$$
{u}=0 \ \ \hbox{on} \ \partial\Omega.\eqno{(5.1)}
$$

\par\noindent
where $\Omega$ is a bounded domain of the Heisenberg group $\mathbb{H}^{n}$ with smooth boundary.
It is known that the Dirichlet problem (5.1) has a discrete spectrum.  The Kohn Laplacian
dates from \cite{Ko}, and the problem (5.1) has been studied, e.g., in \cite{J,NZ}.
We denote its eigenvalues by
$$ 0 < \lambda_{1} \leq \lambda_{2} \leq \cdots \leq \lambda_{k} \cdots  \rightarrow +\infty, $$
and orthonormalize its eigenfunctions $u_{1},\, u_{2},\,\cdots \, \in S^{1,2}_{0}(\Omega)$
so that, $\forall i,j\ge 1$,
$$\left\langle u_{i},u_{j}\right\rangle_{L^{2}}= \int_{\Omega} u_{i} u_{j} dx \, dy \, dt= \delta_{ij}. $$
Here, $S^{1,2}(\Omega)$ denotes the Hilbert space of the functions $u \in L^{2}(\Omega)$ such that $X_{i}(u),\, Y_{i}(u) \in L^{2}(\Omega)$, and $S^{1,2}_{0}$
denotes the closure of $\mathcal{C}^{\infty}_{0}(\Omega)$ with respect to the Sobolev norm
$$ \|u\|^{2}_{S^{1,2}}= \int_{\Omega} (|\nabla_{\mathbb{H}^{n}}u|^{2}+ |u|^{2}) dx\,dy\,dt,$$
with $\nabla_{\mathbb{H}^{n}}u= (X_{1}(u),\cdots, X_{n}(u),Y_{1}(u),\cdots,Y_{n}(u)).$
\par
We shall prove a result similar to Theorem~\ref{PPW euclidean} for the problem (5.1):
\begin{theorem}\label{PPW Heisenberg} For any $k \geq 1$
\begin{enumerate}
\item[(I)] n $\displaystyle{\sum_{i=1}^{k}(\lambda_{k+1}-\lambda_{i})^{2} \le 2\sum_{i=1}^{k}(\lambda_{k+1}-\lambda_{i})\lambda_{i}}$
\item[(II)]
$\displaystyle{\left( \frac{n+1}{nk}\right)\sum_{i=1}^{k} \lambda_{i}- \sqrt{D}} \le \lambda_{k+1} \le  
\left( \frac{n+1}{nk}\right)\sum_{i=1}^{k} \lambda_{i}+ \sqrt{\tilde D_{nk}}$
\end{enumerate}
where ${\tilde D}_{nk}= \left(\left(1 + \frac{1}{n}\right)  \frac{1}{k} \sum_{i=1}^{k}{\lambda_{i}}\right)^2 -
\left(1 + \frac{2}{n}\right)  \frac{1}{k} \sum_{i=1}^{k}{\lambda_{i}^2} \ge 0.$
\end{theorem}

\begin{remark}\label{rk4}
Using the Cauchy--Schwarz inequality $(\sum_{i=1}^{k} \lambda_{i})^{2} \leq k \sum_{i=1}^{k} \lambda_{i}^2$, we deduce from Theorem~\ref{PPW Heisenberg} (II) that
$$ \lambda_{k+1} \leq \left(\frac{1}{k}+\frac{2}{nk}\right)\left(\sum_{i=1}^{k} \lambda_{i}\right)$$
which improves a result of Niu and Zhang \cite{NZ}.
\end{remark}
\begin{proof}
The key observation here is that Lemma \ref {commutation} remains valid for $H=L= -\Delta_{\mathbb{H}^{n}}$ and $G=x_{\alpha}$ or $G=y_{\alpha}$.
Thus we have
\begin{eqnarray}
\nonumber{} \sum_{i=1}^{k}\sum_{\alpha=1}^{n}(\lambda_{k+1}-\lambda_{i})^{2} (\langle[L,x_{\alpha}]u_{i},x_{\alpha}u_{i}\rangle_{L^{2}}&+&\langle[L,y_{\alpha}]u_{i},y_{\alpha}u_{i}\rangle_{L^{2}}) \le \\ \nonumber \sum_{i=1}^{k}\sum_{\alpha=1}^{n}(\lambda_{k+1}-\lambda_{i})(\left\|[L,x_{\alpha}]u_{i}\right\|_{L^2}^{2}&+&\left\|[L,y_{\alpha}]u_{i}\right\|_{L^2}^{2})\ \ \ \ \ \ \ \ (5.2)
\end{eqnarray}
with
\begin{center}
$ \left[L,x_{\alpha}\right]u_{i}=-2 X_{\alpha}(u_{i})$  and  $ \left[L,y_{\alpha}\right]u_{i}=-2 Y_{\alpha}(u_{i}).$
\end{center}
Thus, $$ \sum_{\alpha=1}^{n}\left\|[L,x_{\alpha}]u_{i}\right\|_{L^2}^{2}+\left\|[L,y_{\alpha}]u_{i}\right\|_{L^2}^{2}= 4\int_{\Omega} |\nabla_{\mathbb{H}^{n}}u_{i}|^{2}=4 \lambda_{i}.$$
Now, using the skew-symmetry of $X_{\alpha}$ (resp. $Y_{\alpha})$, we have
$$\int_{\Omega} X_{\alpha}(u_{i})\,x_{\alpha}u_{i}= -\int_{\Omega} u_{i} X_{\alpha}(x_{\alpha}u_{i})= -\int_{\Omega} u_{i}^{2} -\int_{\Omega} X_{\alpha}(u_{i}) x_{\alpha}u_{i}$$
and the same identity holds with $y_{\alpha}$ and $Y_{\alpha}$.  Therefore,
$$ -2 \int_{\Omega} X_{\alpha}(u_{i}) x_{\alpha}u_{i}= -2 \int_{\Omega} Y_{\alpha}(u_{i})y_{\alpha}u_{i}= \int_{\Omega} u_{i}^{2}=1.$$
We put these identities in (5.2) and obtain the first assertion of Theorem~\ref {PPW Heisenberg}.
The second assertion follows as in the 
proof of Theorem~\ref {PPW euclidean}.~\end{proof}

\subsection*{Acknowledgments}

This work was partially supported by US NSF grant DMS-0204059, and was done in large measure while E.~H. was a visiting professor at the Universit\'e Fran\c cois Rabelais.  We also wish to
thank Mark Ashbaugh and Lotfi Hermi for remarks and references.


\end{document}